\documentclass[12pt]{amsart}

\usepackage{amsmath}
\usepackage{amssymb, amsfonts,amscd,verbatim,hyperref,cleveref,graphics}
\usepackage{xspace}
\usepackage{tikz}
\usetikzlibrary{matrix,calc,decorations.pathmorphing,shapes,arrows}

\newtheorem{thm}{Theorem}[section]
\newtheorem{prop}[thm]{Proposition}
\newtheorem{cor}[thm]{Corollary}
\newtheorem{lemma}[thm]{Lemma}
\numberwithin{equation}{section}

\theoremstyle{definition}

\newtheorem{ex}[thm]{Example}
\newtheorem{quest}[thm]{Question}

\def\NN{{\mathbb N}}
\def\RR{{\mathbb R}}

\DeclareSymbolFont{bbold}{U}{bbold}{m}{n}
\DeclareSymbolFontAlphabet{\mathbbold}{bbold}
\def\one{\mathbbold{1}}

\newcommand{\zs}

\newcommand{\Iff}{if\textcompwordmark f\xspace}

\newcommand{\Cl}{\operatorname{Cl}\nolimits}

\newcommand{\Ran}{\operatorname{Ran}\nolimits}
\newcommand{\Int}{\operatorname{Int}\nolimits}

\newcommand{\term}[1]{{\textit{\textbf{#1}}}}   

\newcommand{\goesun}{\xrightarrow{\mathrm{un}}}     

\begin{document}

\title{Topological aspects of order in $C(X)$}

\author{M.~Kandi\' c}
\address{Fakulteta za Matematiko in Fiziko,
Univerza v Ljubljani,
Jadranska ulica 19,
SI-1000 Ljubljana,
Slovenija }
\email{marko.kandic@fmf.uni-lj.si}

\author{A.~Vavpeti\v c}
\address{Fakulteta za Matematiko in Fiziko,
Univerza v Ljubljani,
Jadranska ulica 19,
SI-1000 Ljubljana,
Slovenija }
\email{ales.vavpetic@fmf.uni-lj.si}


\thanks{This research was supported by the Slovenian Research
Agency grants P1-0222-0101, P1-0292-0101, J1-6721-0101 and J1-7025-0101.}

\begin{abstract}
In this paper we consider the relationship between order and topology in the vector lattice $C_b(X)$ of all bounded continuous functions on a Hausdorff space $X$. We prove that the restriction of $f\in C_b(X)$ to a closed set $A$ induces an order continuous operator \Iff $A=\overline{\Int A}.$ This result enables us to easily characterize bands and projection bands in $C_0(X)$ and $C_b(X)$ through the one-point compactification and the Stone-\v Cech compactification of $X$, respectively. With these characterizations we describe order complete $C_0(X)$ and $C_b(X)$-spaces in terms of extremally disconnected spaces. Our results serve us to solve an open question on lifting un-convergence in the case of $C_0(X)$ and $C_b(X)$.
\end{abstract}

\maketitle

\section{Introduction}

The interaction between order and topology on spaces of continuous functions have been studied extensively in the past. Although closed ideals, bands and projection bands in $C(X)$ where $X$ is compact and Hausdorff have nice characterizations through closed sets of $X$, the lattice $C(X)$ is very rarely order complete. Order completeness of the lattice $C(X)$ is equivalent to the fact that $X$ is extremally disconnected \cite{Meyer-Nieberg:91}.

In \Cref{Section: order density} we consider the operator $\Phi_A:C(X)\to C(A)$ which maps continuous functions on $X$ to their restrictions on $A\subseteq X$. Although this operator is a contraction between $C_b(X)$ and $C_b(A)$, in general it is not order continuous. We prove that whenever $X\in T_{3 \frac 1 2}$ order continuity of $\Phi_A$ is equivalent to $A\subseteq \overline{\Int A}$ (see \Cref{OC restriction}).

Although order properties of $C(X)$ where $X$ is a compact Hausdorff space can be found in the main textbooks on vector and Banach lattices, results for $C_b(X)$ where $X$ is ``just" Hausdorff seem to be left out of the main literature. Our second goal is to provide an elementary characterization of closed ideals, bands and projection bands of Banach lattices $C_b(X)$ and $C_0(X)$ through closed sets of $X$ where $X$ is Hausdorff and locally compact, respectively. Main tools we use are order continuity of the restriction operator $\Phi_A$ and the  embeddings of $X$ into $X^+$ or $\beta X$ when $X$ is locally compact Hausdorff space or $T_{3\frac 1 2}$, respectively. There is also a way how to study $C_b(X)$ where $X$ is just Hausdorff. If $X$ is $T_{3 \frac 1 2}$, then bands and projection bands have desired characterizations, and when $X$ is only Hausdorff, desired characterizations can be obtained only for projection bands (see \Cref{counterex}).
Certain results of \Cref{C_0} and \Cref{C_b} are applied in \Cref{extr oc} to characterize order complete vector lattices $C_0(X)$ and $C_b(X)$. We prove that $C_0(X)$ or $C_b(X)$, when $X$ is locally compact Hausdorff or $T_{3 \frac 1 2}$, respectively, is order complete \Iff $X$ is extremally disconnected. This enables us to provide a vector lattice argument that $X\in T_{3\frac 1 2}$ is extremally disconnected \Iff its Stone-\v Cech compactification $\beta X$ is.

In \Cref{lift un} we consider the so-called problem of ``lifting un-convergence" \cite[Question 4.7]{KMT}. An application of results from \Cref{C_0} and \Cref{C_b} provides a positive answer in the special case of Banach lattices $C_0(X)$ and $C_b(X)$.

\section{Preliminaries}

Let $E$ be a vector lattice. A vector subspace $Y$ of $E$ is an \term {ideal} or an \term{order ideal} whenever $0\leq |x|\leq |y|$ and $y\in Y$ imply $x\in Y$. A \term{band} $B$ is an ideal with the property that whenever $0\leq x_\alpha\nearrow x$ and $x_\alpha \in B$ for each $\alpha$ imply $x\in B$. Here, the notation $x_\alpha\nearrow x$ means that the net $(x_\alpha)$ is increasing towards its supremum $x$. In the literature bands are sometimes referred to as \term{order closed ideals}. For a given set $A\subseteq E$ by $A^d$ we denote the set $\{x\in E:\; |x|\wedge |a|=0 \textrm{ for all }a\in A\}.$ The set $A^d$ is called the \term{disjoint complement} of $A$ in $E$. It turns out that $A^d$ is always a band in $E$.
If a band $B$ satisfies $E=B\oplus B^d$, then $B$ is called a \term{projection band}. If every band in $E$ is a projection band, then
$E$ is said to have the \term{projection property}. A vector lattice is said to be \term{order complete} whenever nonempty bounded by above sets have suprema. It is well-known that order complete vector lattices have the projection property. The spaces $L_p(\mu)$ $(0\leq p\leq \infty)$ are order complete whenever $\mu$ is a $\sigma$-finite measure. When $p\neq 0$ or $\infty$, then $L_p(\mu)$ is always order complete. On the other hand, the lattice $C(X)$ is rarely order complete.
If $(f_\alpha)$ is an increasing net in $C(X)$ that is bounded  by above, then $f=\sup_\alpha f_\alpha$ exists in the vector lattice of all functions on $X$. If $f$ is continuous, then $f_\alpha\nearrow f$ in $C(X)$. A sublattice $Y$ of a vector lattice $E$ is said to be \term{regular} if for every subset $A$ of~$Y$, $\inf A$ is the same in $E$ and in $Y$ whenever $\inf A$ exists in~$Y$. It is well-known (see e.g. \cite{Aliprantis:03}) that every ideal is a regular sublattice.
The following result characterizes regular sublattices (see e.g. \cite[Theorem 1.20]{Aliprantis:03}).

\begin{lemma}\label{regular}
  For a sublattice $Y$ of~$E$ the following statements are equivalent.
  \begin{enumerate}
  \item [(a)] $Y$ is regular.
  \item [(b)] If $\sup A$ exists in $Y$ then $\sup A$ exists in $E$ and the  two suprema are equal.
  \end{enumerate}
\end{lemma}

The following proposition directly follows from \Cref{regular}.

\begin{prop}\label{bands in ideals}
Let $Y$ be an ideal in a vector lattice $E$. If $B\subseteq Y$ is a band in $E$, then $B$ is a band in $Y$.
\end{prop}

\begin{proof}
Since every band is an ideal, $B$ is an ideal in $E$, and hence $B=B\cap Y$ is an ideal in $Y$.
Suppose now that a net $(x_\alpha)$ satisfies $0\leq x_\alpha\nearrow x$ in $Y$. By \Cref{regular} we have
$x_\alpha\nearrow x$ in $E$. That $x\in B$ follows from the fact that $B$ is a band in $E$. Hence, $B$ is a band in $Y$.
\end{proof}

In this paper we are concerned with algebras/lattices of continuous functions on a topological space $X$. For a set $A$ of $X$ we will denote by $\Int_XA$ and $\Cl_XA$ the interior and the closure of $A$ in $X$, respectively.  If there is no confusion in which space the interior and the closure are taken, we will simply write $\Int A$ and $\overline A$ instead of $\Int_AX$ and $\Cl_XA$, respectively. If $A$ is an open subset of $X$ and $B$ is a subset of $A$, then $\Int_AB=\Int_XB$.
A topological space is said to be \term{completely regular} for every closed set $F$ in $X$ and $x\in X\setminus F$ there exists a continuous function $f\colon X\to [0,1]$ such that $f(x)=1$ and $f|_F\equiv 0.$ If $X$ is completely regular and Hausdorff, we say that $X$ is $T_{3 \frac 1 2}$ and  we write $X\in T_{3\frac 1 2}$ for short. It is well-known that locally compact Hausdorff spaces are $T_{3\frac 1 2}$. A space $X$ is said to be \term{regular} whenever a point and a closed set which does not contain the point can be separated by disjoint neighborhoods. If disjoint closed sets of $X$ can be separated by disjoint neighborhoods, then $X$ is said to be \term{normal}.

If not otherwise stated all functions are assumed to be continuous.
Since real function rings (resp. algebras) $C(X)$, $C_b(X)$ and $C_0(X)$ are also lattices, we introduce the following notation to distinguish between algebraic and lattice notions of ideals. A ring (resp. algebra) ideal is called an \term{$r$-ideal} (resp. an \term{$a$-ideal}) and an order ideal is called an \term{$o$-ideal}. Since all constant functions on $X$ are in $C_b(X)$ (resp. $C(X)$) the class of $r$-ideals in $C_b(X)$ (resp. $C(X)$) coincides with the class of $a$-ideals.
The ideal $J$ in $C(X)$ is called \term{fixed} whenever all functions from $J$ vanish on a nonempty set. The ideals that are not fixed are called \term{free}. Fixed ideals play an important role in the characterization of compact Hausdorff spaces among $T_{3 \frac 1 2}$-spaces. For the proof of the following Theorem see \cite[Theorem 4.11]{Gillman-Jerison:76}.

\begin{thm}\label{maximals compact}
For $X\in T_{3\frac 1 2}$ the following assertions are equivalent. \begin{enumerate}
\item [(a)] $X$ is compact.
\item [(b)] Every $r$-ideal in $C_b(X)$ is fixed.
\item [(c)] Every maximal $r$-ideal in $C_b(X)$ is fixed.
\end{enumerate}
\end{thm}

Special examples of fixed ideals are the ideals of the form
$$J_F(X):=\{f\in C(X):\; f|_F\equiv 0\}$$ for some subset $F$ of $X$. Obviously $J_F(X)$ is an $a$-ideal and an $o$-ideal which is closed in $C(X)$ in the topology of pointwise convergence. It should be obvious that we always have
$J_F(X)=J_{\overline F}(X).$
For $F\subseteq  X$ we denote the set $\{f\in C_0(X):\; f|_F\equiv 0\}$ by $J_{F}^{0}(X)$.

If $ X$ is a compact Hausdorff space there exists a characterisation
of closed ideals, bands and projective bands in $C( X)$.
We start with a characterization of closed ideals (see e.g. \cite[Proposition 2.1.9]{Meyer-Nieberg:91}).

\begin{thm}\label{CharOfClosId}
For every closed subspace $J$ of $C( X)$ where $ X$ is a compact Hausdorff space the following assertions are equivalent.
\begin{enumerate}
\item [(a)] $J$ is an $r$-ideal.
\item [(b)] $J$ is an $o$-ideal.
\item [(c)] $J=J_F(X)$ for some closed set $F\subseteq  X$.
\end{enumerate}
\end{thm}

Bands and projection bands are characterized as follows (see e.g. \cite[Corollary 2.1.10]{Meyer-Nieberg:91}).

\begin{thm}\label{CharOfBands}
For a subspace $B$ of $C(X)$ where $ X$ is a compact Hausdorff space the following assertions hold.
\begin{enumerate}
\item [(a)] $B$ is a band \Iff $B=J_F(X)$ for some  set $F\subseteq X$ which is a closure of some open set.
\item [(b)] $B$ is a projection band \Iff $B=J_F(X)$ for some clopen set $F\subseteq X$.
\end{enumerate}
\end{thm}

\begin{cor}
If $X$ is a locally compact noncompact Hausdorff space, then $C_0(X)$ is not a band in $C(X^+)$.
\end{cor}

A closed set which is a closure of some open set is the closure of its interior. Indeed, if $F=\overline{U}$ where $U$ is an open set, then $U\subseteq \Int F$, so that $F=\overline{U}\subseteq \overline{\Int F}\subseteq F$. Thus, a closed ideal $J$ in $C(X)$ where $X$ is a compact Hausdorff space is a band \Iff  $J=J_F(X)$ for some closed set $F$ in $X$ with $F=\overline{\Int F}$.

For unexplained facts about vector lattices and operators acting on them we refer the reader to \cite{Luxemburg:71}.

\section{Order density and completely regular spaces}\label{Section: order density}

Complete regularity of the space basically means that we can separate points from closed sets by continuous functions. If we can separate only points by continuous functions, i.e., for different points $x,y\in X$ there is a continuous function $f\colon X\to [0,1]$ with $f(x)=0$ and $f(y)=1$, then $X$ is called \term{functionally Hausdorff}.

If $X$ is not functionally Hausdorff, then there exist $x,y\in X$ such that $f(x)=f(y)$ for all $f\in C(X)$. In this case we have  $J_{\{x\}}(X)=J_{\{y\}}(X)$, so that the mapping $x\mapsto J_{\{x\}}(X)$ is not one-to-one. If $X\in T_{3\frac 1 2}$, then for closed sets $F$ and $G$ in $X$ we have $J_F(X)=J_G(X)$ \Iff $F=G$. Indeed, if $F\neq G$, then there exists $x\in X$ such that $x\in F\setminus G$ or $x\in G\setminus F$. Without any loss of generality we can assume that the former happens. Then one can find $f:X\to [0,1]$ with $f(x)=1$ and $f\equiv 0$ on $G$. Then $f\in J_G(X)=J_F(X)$ implies $f(x)=0$ which is a contradiction.

Also, when $F$ is a closed set in $X\in T_{3\frac 1 2}$, the disjoint complement of $J_F(X)$ satisfies $J_F(X)^d=J_{\overline{X\setminus F}}(X).$ Indeed, if $f\in J_{\overline{X\setminus F}}(X)$, then $f\equiv 0$ on $X\setminus F$, so that $|f|\wedge |g|=0$ for all $g\in J_F(X)$. Hence $f\in J_F(X)^d.$ If $f\notin J_{\overline{X\setminus F}}(X)$, there is
$x\in X\setminus F$ such that $f(x)\neq 0$. Take any $g\in C_b(X)$ with $g(x)=1$ and $g\equiv 0$ on $F$. Then $g\in J_F(X)$ and $|f|\wedge |g|\neq 0$, i.e., $f\notin J_F(X)^d.$

In this section we are interested in the so-called ``restriction operator" defined as follows. If $A$ is a subset of $X$, then the restriction $f|_A$ of $f\in C(X)$ is continuous on $A$. The mapping $\Phi_A(f)= f|_A$ is called the \term{restriction operator}.
Also, if $f\in C_b(X)$, then $f|_A\in C_b(A)$ and the restriction operator is a contractive lattice homomorphism with respect to the $\|\cdot\|_\infty$ norms on $C_b(X)$ and $C_b(A)$.

If $X$ is normal and $A$ is closed in $X$, then the restriction operator is always surjective. The details follow.

\begin{thm}\label{Tietze and Urysohn}
The following assertions are equivalent.
\begin{enumerate}
\item [(a)] $X$ is normal.
\item [(b)] For any closed subset $A$ in $X$ the restriction operator $\Phi_A\colon C_b(X)\to C_b(A)$ is surjective.
\item [(c)] For any disjoint closed sets $A$ and $B$ in $X$ there is    $f\in C_b(X)$ such that $f|_A\equiv 1$ and $f|_B\equiv 0.$
\end{enumerate}
\end{thm}

It should be noted that (b) and (c) are precisely the statements of Tietze extension theorem and Urysohn's lemma, respectively.

If $f_\alpha\nearrow f$ in $C(X)$, then for each set $A\subseteq X$ we have
$f_\alpha|_A\nearrow$ and $f_\alpha|_A\leq f|_A$ for each $\alpha$.
Pick $\lambda\in\mathbb R$. If $f_\alpha|_A\equiv \lambda$ for each $\alpha$, then $f|_A\equiv \lambda$.
In general we cannot expect $f_\alpha|_A\nearrow f|_A$, since $\Phi_A$ is not order continuous.
We will prove that for a closed set $A\subseteq X$ the operator $\Phi_A\colon C(X)\to C(A)$ is order continuous \Iff $A=\overline{\Int A}$ (see \Cref{OC restriction}). To prove this result we first need the following lemma which deals with a special case.

\begin{prop}\label{supst na odprto}
Let $X\in T_{3\frac 1 2}$ and $A\subseteq X$ open subset.
\begin{enumerate}
\item [(a)] If $f_\alpha\nearrow f$ in $C(X)$, then $f_\alpha|_{A}\nearrow f|_{A}$ in $C(A)$.
\item [(b)] If $f_\alpha\nearrow f$ in $C_b(X)$, then $f_\alpha|_{A}\nearrow f|_{A}$ in $C_b(A)$.
\end{enumerate}
\end{prop}

\begin{proof}
(a) Let $f_\alpha\nearrow f$ in $C(X)$.
Suppose there exists $g\colon A\to\RR$ such that $f_\alpha|_A \le g<f|_A$ for all $\alpha$.
There exists $x_0\in A$ such that $g(x_0)< f(x_0)$ and because $f$ and $g$ are continuous and $X$ is regular, there exists
an open set $U\subseteq \Cl_X(U)\subseteq A$ such that $g(x)<f(x)$ for all $x\in U$. Because $X\in T_{3\frac 1 2}$
there exists $\lambda\colon X\to [0,1]$ such that $\lambda(x_0)=1$ and $\lambda(X\setminus U)=0$.
We define a function $\lambda g\colon X\to \RR$ as
$$(\lambda g)(x)=\left\{\begin{array}{ccl}
 \lambda(x)g(x)& : &x\in A\\
 0& : &x\in  X\setminus \Cl_XU
 \end{array}\right..$$
Since $\{A,X\setminus \Cl_X(U)\}$ is an open cover of $X$ and $\lambda\equiv 0$ on $X\setminus \Cl_XU$ we conclude that $\lambda g$ is continuous on $X$.
Then $h\colon X\to \RR$ defined by $h(x)=(\lambda g)(x)+(1-\lambda(x))f(x)$ is continuous
and $f_\alpha\le h < f$ for all $\alpha$. This is a contradiction as $f_\alpha\nearrow f$.

(b) Suppose $f_\alpha\nearrow f$ in $C_b(X)$. Since $C_b(X)$ is an ideal in $C(X)$, \Cref{regular} implies $f_\alpha\nearrow f$ in $C(X)$, so that by (a) we have $f_\alpha|_A\nearrow f|_A$ in $C(X)$. If $f_\alpha|_A\leq g\leq f|_A$ for some $g\in C_b(A)$, then $f_\alpha|_A\nearrow f|_A$ in $C(X)$ implies $g=f|_A$, and since $f|_A\in C_b(A)$ we have $f_\alpha|_A\nearrow f|_A$ in $C_b(A)$.
\end{proof}

In \Cref{Spust na TFC} we will see that there is an example of a functionally Hausdorff space $X$ such that for every open set $A$ in $X$ the restriction operator $\Phi_A:C(X)\to C(A)$ is order continuous.

\begin{thm}\label{OC restriction}
%
%
Let $X\in T_{3\frac 1 2}$ and $A\subseteq X$ a closed set. The following is equivalent:
\begin{enumerate}
\item [(a)] $A= \overline{\Int A}$.
\item [(b)] If $f_\alpha\nearrow f$ in $C(X)$, then $f_\alpha|_{A}\nearrow f|_{A}$ in $C(A)$.
\item [(c)] If $f_\alpha\nearrow f$ in $C_b(X)$, then $f_\alpha|_{A}\nearrow f|_{A}$ in $C_b(A)$.
\end{enumerate}
\end{thm}

\begin{proof}
$(a)\Rightarrow (b)$ Suppose $A=\overline{\Int A}$ and let $f_\alpha\nearrow f$ in $C(X)$.
Let $g\colon A\to\RR$ be such that $f_\alpha|_{A}\le g\le f|_A$ for all $\alpha$.
By \Cref{supst na odprto}, $f_\alpha|_{\Int A}\nearrow f|_{\Int A}$ in $C(\Int A)$,
hence $g|_{\Int A}=f|_{\Int A}$. Because of continuity of $f$ and $g$ (on $A$)
and $A=\overline{\Int A}$ we get $g=f|_{A}$.

$(b)\Rightarrow (a)$ Suppose $A\nsubseteq \overline{\Int A}$. There exists $x_0\in A\setminus \overline{\Int A}$.
Because $X\in T_{3\frac 1 2}$ there exists $\widetilde\psi\colon X\to [-1,1]$ such that
$\widetilde\psi(x_0)=-1$ and $\widetilde\psi(\overline{\Int A})=1$. Let $\psi\colon X\to [0,1]$
be defined as $\psi=\max\{\widetilde\psi,0\}$. Then $U={\widetilde\psi}^{-1}([-1,0))\cap A$ is an open set
in $A$. Because $A$ is closed, $\overline U\subseteq A$ and $\overline U\subseteq {\widetilde\psi}^{-1}([-1,0])$,
hence $x_0\in U\subseteq \overline U\subseteq A\setminus {\overline{\Int A}}$.
Let
${\mathcal M}=\{f\colon X\to \RR:\; f|_{\overline U}\equiv 0, f|_{{\overline{\Int A}}}\equiv 1\}$. By the construction we have $\psi\in \mathcal M$, so that $\mathcal M$ is nonempty.

\emph{Claim 1:} If $f\le g$ for all $f\in{\mathcal M}$ then $g(x)\ge 1$ for all $x\in X\setminus \overline U$. Indeed, pick $x\not\in\overline U$. Because $X\in T_{3\frac 1 2}$ there exists $\varphi\colon X\to [0,1]$
such that $\varphi\equiv 0$ on $\overline U$ and $\varphi(x)=1$. Then $f:=\min\{\varphi+\psi,1\}\in{\mathcal M}$
and $f(x)=1$, which implies $g(x)\ge f(x)=1$.

\emph{Claim 2:} The set $X\setminus \overline U$ is dense in $X$. Indeed, otherwise there exists an open set $V\subseteq X$
such that $V\cap (X\setminus \overline U)=\emptyset$.
Then $V\subseteq \overline U\subseteq A\setminus \overline{\Int A}.$
This is a contradiction, since from $V\subseteq A$ and $V$ open follows $V\subseteq \Int A$.

Let us show that $(f)_{f\in{\mathcal M}}\nearrow 1$ in $C(X)$. Suppose  some function
$g\colon X\to \RR$ satisfies $f\le g\le 1$ for all $f\in{\mathcal M}$.
Claim 1 yields that $g\geq 1$ on $X\setminus \overline{U}$, so that together with $g\leq 1$ we conclude $g|_{X\setminus \overline U}\equiv 1.$ Since by Claim 2 the set $X\setminus \overline U$ is dense in $X$, we have $g\equiv 1$ on $X$. Therefore, $(f)_{f\in\mathcal M}\nearrow 1$ in $C(X)$.

Let us show that $(f|_A)_{f\in{\mathcal M}}$ does not converge in order to $1$ in $C(A)$.
If $(f|_A)_{f\in{\mathcal M}}\nearrow 1$, by \Cref{supst na odprto}
$(f|_U)_{f\in{\mathcal M}}\nearrow 1$. But $f|_U\equiv 0$ for all $f\in{\mathcal M}$.

(b) $\Rightarrow$ (c) Suppose $f_\alpha\nearrow f$ in $C_b(X)$. Then $f_\alpha\nearrow f$ in $C(X)$, so that $f_\alpha|_A\nearrow f|_A$ in $C(A)$. Since $f|_A$ is bounded, we have $f_\alpha|_A\nearrow f|_A$ in $C_b(A)$.

(c) $\Rightarrow$ (b) Since for each $g\in C(X)$ we have $g\wedge n\one\nearrow_n  g$ and since for each $n\in\mathbb N$ we have
$f_\alpha|_A\wedge n\one\nearrow_\alpha f|_A\wedge n\one,$ we conclude
\begin{align*}
f|_A&=\sup_{n\in\mathbb N}\left(f|_A\wedge n\one\right)=\sup_{n\in\mathbb N}\sup_{\alpha}\left(f_\alpha|_A\wedge n\one\right)\\
&=\sup_\alpha\sup_{n\in\mathbb N}\left(f_\alpha|_A\wedge n\one\right)=\sup_{n\in\mathbb N}f_\alpha|_A.
\end{align*}
\end{proof}

\begin{cor}\label{OC continuity closed set}
Let $A$ be a closed set in $X\in T_{3\frac 1 2}.$ Then the restriction operator $\Phi_A$ is order continuous \Iff $A=\overline{\Int A}.$
\end{cor}

Suppose now that $X$ is compact and Hausdorff; hence normal and so $X\in T_{3 \frac 1 2}$. If $A$ is closed in $X$, then the kernel of $\Phi_A\colon C(X)\to C(A)$ is precisely the fixed ideal $J_A(X)$. In this special case, by \Cref{CharOfBands} and \Cref{OC continuity closed set} the restriction operator $\Phi_A$ is order continuous \Iff $J_A(X)$ is a band in $C(X)$.

When $X$ is merely normal, the restriction operator $\Phi_A$ is surjective by \Cref{Tietze and Urysohn}. By \cite[Theorem 18.13]{Luxemburg:71} $\Phi_A\colon C_b(X)\to C_b(A)$ is order continuous \Iff $\ker \Phi_A=J_A(X)$ is a band in $C_b(X)$. In \Cref{bands in Tychonoff} we will prove that for $X\in T_{3\frac 1 2}$ bands in $C_b(X)$ are precisely closed ideals in $C_b(X)$ of the form $J_A(X)$ where $A$ is a closed set satisfying $A=\overline{\Int A}.$

If we only assume $X\in T_{3\frac 1 2}$, by \Cref{Tietze and Urysohn} the restriction operator $\Phi_A$ is not surjective in general. When $A=\overline{\Int A}$, then \Cref{bands in Tychonoff} implies $J_A(X)$ is a band in $C_b(X)$ and $\Phi_A\colon C_b(X)\to \Ran \Phi_A\subseteq C_b(A)$ is an order continuous lattice homomorphism by \cite[Theorem 18.13]{Luxemburg:71}. Therefore, if $f_\alpha\nearrow f$ in $C_b(X)$, then $f_\alpha|_A\nearrow f|_A$ in the sublattice $\Ran \Phi_A$. Since $\Ran\Phi_A$ is order dense in $C_b(A)$ (see \Cref{order density in C_b} below), then it is also a regular sublattice of $C_b(A)$ and hence by \Cref{regular} we have $f_\alpha|_A\nearrow f|_A$ in $C_b(A)$.

\begin{prop}\label{order density in C_b}
Let $X\in T_{3 \frac 1 2}$, $A\subseteq X$, and $\Phi_A\colon C_b(X)\to C_b(A)$ the restriction operator.
Then $\Ran\Phi_A$ is order dense in $C_b(A)$.
\end{prop}

\begin{proof}
Pick a nonzero nonnegative function $f\in C(A)$. Let $a\in A$ be such that $f(a)>0$.
Because $f$ is continuous there exists an open set $U$ in $A$ such that $f(x)>\tfrac{f(a)}2$ for all $x\in U$.
Let $V$ be open in $X$ such that $V\cap A=U$. Because $X\in T_{3 \frac 1 2}$ there exists
$g\colon X\to[0,\tfrac{f(a)}2]$ such that $g(a)=\tfrac{f(a)}2$ and $g\equiv 0$ on $V^C$.
For $x\in U$ we have $f(x)>\tfrac{f(a)}2\geq g(x)$ and for $x\in A\setminus U$ we have $f(x)\ge 0=g(x)$.
Hence $0\leq g|_A\leq f$.
\end{proof}

%

\section{Bands in $C_0(X)$} \label{C_0}

In this section we are interested in closed ideals of the Banach lattice algebra $C_0(X)$ where $X$ is a locally compact Hausdorff space. We characterize closed ideals, bands and projection bands of $C_0(X)$ (see \Cref{ideals in c_0} and \Cref{bands in c_0}). The characterizations are pretty much the same as the characterization of closed ideals, bands and projection bands of the algebra of all continuous functions on a given compact Hausdorff space (see \Cref{CharOfClosId} and \Cref{CharOfBands}). Our arguments presented in this section are based on the embedding of $X$ into its one-point compactification $X^+$. Before we proceed to our results of this section we briefly recall basic facts about the topology on $X^+$.

Let $X$ be a locally compact Hausdorff space and pick any object $\infty$ which is not in $X$. It is well-known that the family $\tau$ of all open sets in $X$ together with the sets in $X^+$  containing $\infty$ whose complements are compact in $X$ is a topology on $X^+$ and the space $(X^+,\tau)$ is a compact Hausdorff space. The space $X^+$ is called the \term{one-point compactification} of $X$. By \cite[8.3 p.245]{Dugundji:66} the embedding $X\hookrightarrow X^+$ is an open mapping; in particular, $X$ is an open subset of $X^+$. If $X$ is noncompact, then $X$ is dense in $X^+$, and if $X$ is compact, then $X$ is clopen in $X^+$ and $\infty$ is isolated in $X^+$.

Suppose now that $X$ is not compact  and $f\in C_0(X)$. Then $f$ has the unique extension $\widetilde f\in C(X^+)$ on $X^+$ given by
$$\widetilde f(x)=\left\{\begin{array}{ccl}
f(x) &: &x\in  X\\
0 &: &x=\infty
\end{array}
\right..$$
The mapping $f\mapsto \widetilde f$ is an algebra and lattice isometric isomorphism between $C_0(X)$ and the closed maximal lattice and algebra ideal $J_{\{\infty\}}(X^+)$ in $C(X^+)$.
For brevity purposes we will write $J_{\{\infty\}}$ instead of $J_{\{\infty\}}(X^+).$
The closed lattice (resp. algebra) ideals of $C_0(X)$ are therefore in bijective correspondence with the closed lattice (resp. algebra) ideals of $C(X^+)$ that consist of functions that vanish at infinity.

As was already mentioned closed ideals are characterized similarly as the closed ideals in $C(K)$-spaces.

\begin{thm}\label{ideals in c_0}
Let $X$ be a locally compact Hausdorff space and $J$ a closed subspace in $C_0(X)$.
The following assertions are equivalent:
\begin{enumerate}
\item [(a)] $J$ is an $a$-ideal in $C_0(X)$.
\item [(b)] $J$ is an $o$-ideal in $C_0(X)$.
\item [(c)] $J=J_F^0(X)$ for some closed set $F$ in $X$.
\end{enumerate}
\end{thm}

\begin{proof}
When $X$ is compact the conclusion follows from \ref{CharOfClosId}, so that we may assume $X$ is not compact.
By \cite[Theorem 1.4.6]{Kaniuth:09} (a) and (c) are equivalent. It is trivial that (c) implies (b).

To finish the proof we prove that (b) implies (a). Since $C_0(X)$ is isometrically lattice isomorphic to  $J_{\{\infty\}}$, $J$ is isometrically lattice isomorphic to some closed lattice ideal $J'=J_{F'}(X^+)$ of $C(X^+)$ for some closed subset $F'\subseteq X^+$ which contains the point $\infty$. By \Cref{CharOfClosId} $J'$ is also an $a$-ideal, from where it follows that $J=\{f|_X:\, f\in J'\}$ is an $a$-ideal in $C_0(X)$.
\end{proof}

We proceed by a characterization of bands and projection bands in $C_0(X)$. Our proofs presented here rely on the embedding $X\hookrightarrow X^+$.

\begin{thm}\label{bands in c_0}
Let $X$ be a locally compact Hausdorff space and $J$ a closed ideal in $C_0(X)$.
\begin{enumerate}
\item [(a)] $J$ is a band  \Iff $J=J_F^0(X)$ for some closed set $F$ in $X$ with  $F=\overline{\Int F}$.
\item [(b)] $J$ is a projection band \Iff $J=J_F^0(X)$ for some clopen set $F$ in $X$.
\end{enumerate}
\end{thm}

\begin{proof}
(a) Let $F=\Cl_{X}(\Int F)$ and let $0\leq f_\alpha\nearrow f$, where $f_\alpha\in J_F^0(X)$ and $f\in C_0(X)$.
Denote  by $\widetilde f_\alpha\colon X^+\to \RR$ the extension of $f_\alpha$ on $X^+$.
Then $0\leq \widetilde f_\alpha\nearrow \widetilde f$ in $J_{\{\infty\}}$, so that by \Cref{regular} $\widetilde f_\alpha\nearrow \widetilde f$ in $C(X^+)$.
If $\Cl_{ X^+}(\Int F)=F$, then $J_F(X^+)$ is a band in $C(X^+)$, hence $\widetilde f\in J_F(X^+)$ and $f\in J^0_F(X)$.
If $\Cl_{ X^+}(\Int F)=F\cup\{\infty\}$, then $\Cl_{ X^+}(\Int (F\cup\{\infty\}))=F\cup\{\infty\}$, hence
$J_{F\cup\{\infty\}}(X^+)$ is a band in $C(X^+)$,
$\widetilde f\in J_{F\cup\{\infty\}}(X^+)$, and $f\in J^0_F(X)$.

If $\Cl_{X}(\Int F)\subsetneq F$, then there exists $x\in F\setminus\Cl_{ X}(\Int F)$.
Let $U$ be an open set such that $\infty\in U\subseteq \overline U\subseteq\{x\}^C$.
We claim that $x\in F\cup \overline U\setminus \Cl_{ X^+}(\Int(F\cup\overline U))$.
Since $x\notin \Cl_X(\Int F)$ there exists an open neighborhood $W_1$ of $x$ such that $W_1\cap \Int F=\emptyset$.
For $W_2:=X^+\setminus\overline U$ we have $x\in W_2$ and
\begin{align*}
W_1\cap W_2\cap \Int(F\cup \overline U)&\subseteq W_1\cap W_2\cap (F\cup \overline U)\\
&=(W_1\cap W_2\cap F)\cup (W_1\cap W_2\cap \overline U)\\
&=W_1\cap W_2\cap F.
\end{align*}
Since the set $W_1\cap W_2\cap \Int(F\cup \overline U)$ is open, we have \begin{align*}
W_1\cap W_2\cap \Int(F\cup \overline U)&\subseteq \Int(W_1\cap W_2\cap F)=W_1\cap W_2\cap \Int F\\
&=(W_1\cap \Int F)\cap W_2=\emptyset.
\end{align*}
Therefore, $W_1\cap W_2$ is an open neighborhood of $x$ which does not intersect $\Int(F\cup \overline U)$, hence $x\notin \Cl_{X^+}\Int(F\cup \overline U).$

Because $F\cup\overline U$ is closed in $X^+$ and $F\cup\overline U\ne \Cl_{ X^+}(\Int(F\cup\overline U))$,
there exists a net $(f_\alpha)_{\alpha\in\Lambda}
\subseteq J_{F\cup\overline U}(X^+)$ such that $f_\alpha\nearrow f\in C( X^+)\setminus J_{F\cup\overline U}(X^+)$.
Then $f_{\alpha}|_{ X}\in J^0_F(X)$ and $f_{\alpha}|_{ X}\nearrow f|_{ X}\in C_0(X)$. 
Since $f_\alpha|_U\equiv 0$, \Cref{supst na odprto} implies  $f|_U\equiv 0$, so that by continuity of $f$ we have $f|_{\overline U}\equiv 0$. Now it follows that $f|_F\not\equiv 0$,  hence $J^0_F(X)$ is not a band in $C_0(X)$.

(b) If $F$ is clopen, then $X\setminus F$ is clopen as well, so that the characteristic functions $\chi_F$ and $\chi_{X\setminus F}$ are continuous on $X$.
Then $f_1:=f\chi_{X\setminus F}\in J^0_F(X)$, $f_2:=f\chi_F\in J^0_F(X)^d$ and $f=f_1+f_2$ imply $C_0(X)=J_F^0(X)\oplus J_F^0(X)^d.$

Suppose now that $J$ is a projection band in $C_0(X)$. By (a) we have $J=J_F^0(X)$ for some closed set $F$ in $X$. We claim $X\setminus F$ is closed in $X$. Otherwise there is $x_0\in F\cap \overline{X\setminus F}.$ Pick any function $f\in C_0(X)$ with $f(x_0)=1.$ Since $J$ is a projection band, we have $C_0(X)=J_F^0(X)\oplus J_{\overline{X\setminus F}}^0(X)$. Hence, $f=f_1+f_2$ for some $f_1\in J_F^0(X)$ and $f_2\in J_{\overline{X\setminus F}}^0(X)$, so that $f(x_0)=0$. This contradiction shows $X\setminus F$ is closed. Hence $F$ is clopen and the proof is finished.
\end{proof}

Closed ideals in $C_0(X)$ which are projection bands in $C(X^+)$ are more interesting. Not only that the set $F$ which defines a given closed ideal is clopen, its complement in $X$ is compact.

\begin{cor}\label{C_0 projection in C}
Let $ X$ be a locally compact Hausdorff space and $J_F^0(X)$ a closed ideal in $C_0(X)$. Then $J_F^0(X)$ is a projection band in $C(X^+)$ \Iff $F$ is closed in $X$ and $X\setminus F$ is compact.
\end{cor}

\begin{proof}
Suppose first that $J_F^0(X)$ is a projection band in $C(X^+)$. Then $F\cup \{\infty\}$ is clopen in $X^+$, from where it follows that $F$ is clopen in $X$. Since
$$X\setminus F=X^+\setminus (F\cup\{\infty\})$$ is also closed in $X^+$, the set $X\setminus F$ is compact.

Suppose now that $F$ is closed  in $X$ and $X\setminus F$ is compact. Since $F$ is closed in $ X$ it follows that $F\cup\{\infty\}$ is closed in $ X^+$. Since $ X\setminus F$ is compact, $F\cup\{\infty\}$ is also open in $ X^+$.
Therefore, $F\cup\{\infty\}$ is clopen in $ X^+$, so that
$J_F^0(X)\cong J_{F\cup\{\infty\}}(X^+)$ is a projection band in $C(X^+)$ by \Cref{CharOfBands}(b).
\end{proof}
%
%
%



Suppose $J$ is an ideal in $C_0(X)$. If $J$ is a band in $C(X^+)$, then $J$ is a band in $C_0(X)$ by \Cref{bands in ideals}. This fact can be also proved directly with results of this section without referring to general result on vector lattices. Indeed, if $J_F^0(X)$ in $C_0(X)$ is band in $C(X^+)$, then \Cref{CharOfBands} implies that $F\cup\{\infty\}=\Cl_{ X^+}(\Int(F\cup\{\infty\}))$.
  Since we have $\Int_XF=\Int_{X^+}F=\Int_{X^+}(F\cup\{\infty\})$, we have
  \begin{align*}
  \Cl_X\Int_XF&=X\cap \Cl_{X^+}\Int_XF=X\cap \Cl_{X^+}\Int_{X^+}(F\cup\{\infty\})\\
  &=X\cap (F\cup\{\infty\})=F.
  \end{align*}
  By \Cref{bands in c_0} $J_F^0(X)$ is a band in $C_0( X)$.

\section{Bands in $C_b( X)$} \label{C_b}

In this section we are interested in extending \Cref{CharOfClosId} and \Cref{CharOfBands} to the case of the algebra $C_b(X)$ where $X$ is a Hausdorff space. To have a nice characterization of closed ideals of $C_b(X)$ in terms of closed sets from $X$ is too optimistic even when $X\in T_{3 \frac 1 2}.$ If for some $X\in T_{3 \frac 1 2}$ every closed ideal in $C_b(X)$ is of the form $J_F(X)$ for some closed set $F$ in $X$, then every closed ideal in $C_b(X)$ is fixed, so that by \Cref{maximals compact} $X$ is compact. Hence, if $X$ is a noncompact space, then there exist closed ideals in $C_b(X)$ which are not of the form $J_F(X)$ for some closed set $F$ in $X$. However, if $X\in T_{3\frac 1 2}$ then bands and projection bands in $C_b(X)$ have similar characterizations as bands in $C(K)$-spaces (see \Cref{bands in Tychonoff}).  If $X$ is merely Hausdorff, then only projection bands in $C_b(X)$ have a similar characterization as those in $C(K)$-spaces (see \Cref{proj band in hausdorff}). Unfortunately, \Cref{counterex} shows that there exists a functionally Hausdorff space $X$ and a band $J_F(X)$ in $C_b(X)$ with $\overline{\Int F}\neq F$.

The following topological facts are critical in our study of bands and projection bands in $C_b(X)$.

\begin{itemize}
    \item For every space $X\in T_{3\frac 1 2}$ there exists a compact Hausdorff space $\beta X$ with the property that the inclusion $\triangle: X\hookrightarrow \beta X$ is an embedding and $\beta X$ contains $X$ as a dense subset. Furthermore, every function $f\in C_b(X)$ can be uniquely extended to the continuous function $f^\beta\in C(\beta X)$.
    \item For every Hausdorff space $X$ there exists $X/_\sim\in T_{3\frac 1 2}$ and a continuous surjective mapping $\tau\colon X\to X/_\sim$ such that the mapping $\Phi_\tau\colon C_b(X/_\sim)\to C_b(X)$ defined as $\Phi_\tau(\widetilde f)= \widetilde f\circ \tau$ is an isometric isomorphism between Banach lattice algebras $C_b(X/_\sim)$ and $C_b(X)$.
        We will recall the construction later in this section.
\end{itemize}

It is well-known that the compact Hausdorff space $\beta X$ is called the \term{Stone-\v Cech compactification} of $X$. By \cite[Theorem 3.70]{Aliprantis-Border:06} the inclusion $\triangle\colon X\hookrightarrow \beta X$ is an open mapping whenever $X$ is locally compact and Hausdorff; in particular, in this case $X$ is open in $\beta X$.

The mapping $\beta\colon C_b(X)\to C(\beta X)$ defined as $\beta(f)= f^\beta$ is an isometric isomorphism of Banach algebras. By \cite[Theorem 1.6]{Gillman-Jerison:76} the mapping $\beta$ is also a lattice isomorphism. Therefore, $J$ is a closed order/algebra ideal (resp. band, projection band) in $C_b(X)$ \Iff $\beta(J)$ is a closed ideal (resp. band, projection band) in $C(\beta X)$. Similarly, given a Hausdorff space $X$, the aforementioned space $X/_\sim$ is $T_{3 \frac 1 2}$, and the isomorphism between Banach algebras $C_b(X)$ and $C_b(X/_\sim)$ is also a lattice isomorphism, again by \cite[Theorem 1.6]{Gillman-Jerison:76}.
Therefore, starting with a subspace $J$ of $C_b(X)$ where $X$ is Hausdorff we deduce that $J$ is a closed order/algebra ideal in $C_b(X)$ \Iff
$\Phi_\tau^{-1}(J)$ is a closed order/algebra ideal in $C_b(X/_\sim)$ \Iff
$(\beta\circ \Phi_\tau^{-1})(J)$ is a closed order/algebra ideal in $C(\beta X)$. Similar statements hold for bands and projection bands.

Therefore, for any closed ideal $J$ in $C_b(X)$ where $X\in T_{3\frac 1 2}$ there exists a closed subset $F$ in $\beta X$ such that
$J=\{f|_X:\; f\in J_F(\beta X)\}.$ Although this description is sufficient to have some information on closed ideals of $C_b(X)$, it is not the best one since $F$ is a subset of $\beta X$ and not of $X$.

The following two results are needed in the proof of \Cref{bands in Tychonoff}. Although the proof of the first one is quite trivial we include it for the sake of completeness.

\begin{lemma}\label{fixed in Stone-Chech}
Let $F$ be a closed set in $X\in T_{3\frac 1 2}$. Then
$\beta(J_F(X))=J_{\overline F}(\beta X)$ where $\overline F$ denotes the closure of $F$ in $\beta X$.
\end{lemma}

\begin{proof}
If $f\in J_F(X)$, then $f\equiv 0$ on $F$, and hence by continuity $f^\beta\equiv 0$ on $\overline{F}$. For the opposite inclusion, pick $f^\beta\in J_{\overline{F}}(\beta X)$. Then $f\equiv 0$ on $F$, so that $f^\beta \in \beta(J_F(X)).$
\end{proof}

\begin{prop}\label{spust zaprtja notranjosti}
Let $X$ be a dense subset of a topological space $Y$. If a closed set $F\subseteq Y$ satisfies $F=\Cl_Y\Int_Y F$, then the closed set $F\cap X$ in $X$ satisfies $F\cap X=\Cl_X \Int_X (F\cap X).$
\end{prop}

\begin{proof}
Suppose $F=\Cl_{Y}\Int_Y F$. Since $F\cap X$ is closed in $X$, we obviously have $\Cl_X\Int_X(F\cap X)\subseteq F\cap X$. To prove the opposite inclusion pick any $x\in F\cap X$. We will prove that an arbitrary open neighborhood in $X$ of $x$ intersects $\Int_X(F\cap X)$.

Let $U$ be an arbitrary open neighborhood in $X$ of $x$. There exists an open neighborhood $V$ in $Y$ of $x$ such that $U=V\cap X$. From
$F=\Cl_{Y}\Int_Y F$ it follows $V\cap \Int_Y F\neq \emptyset$, so that from the density of $X$ in $Y$ we conclude
$V\cap \Int_Y F\cap X\neq \emptyset$. Since $V\cap \Int_Y F\cap X$ is open in $X$ and is obviously contained in $F\cap X$, we have
$$\emptyset\neq U\cap \Int_Y F=V\cap \Int_Y F\cap X\subseteq \Int_X(F\cap X).$$
In particular, $U$ intersects $\Int_X(F\cap X)$ and the proof is finished.
\end{proof}

\Cref{spust zaprtja notranjosti} can be applied in two interesting situations when $X$ is locally compact Hausdorff space and $Y=X^+$ or $X\in T_{3\frac 1 2}$ and $Y=\beta X$. We will use \Cref{spust zaprtja notranjosti} for the latter situation in the proof of \Cref{bands in Tychonoff}.


\begin{thm}\label{bands in Tychonoff}
Let $X\in T_{3\frac 1 2}$ and let $B$ be a closed ideal in $C_b(X)$
\begin{enumerate}
\item [(a)] $B$ is a band in $C_b(X)$ \Iff $B=J_F(X)$ for some closed set $F$ in $X$ with $F=\overline{\Int F}.$
\item [(b)] $B$ is a projection band \Iff $B=J_F(X)$ for some clopen subset $F$ of $ X.$
\end{enumerate}
\end{thm}

\begin{proof}
 (a) Suppose $B$ is a band in $C_b(X)$. By \Cref{bands in c_0} there exists a closed set $F\subseteq \beta X$ such that $F=\Cl_{\beta X}{\Int_{\beta X} F}$ and $\beta(B)=J_F(\beta X)$.

 We claim that $B=J_{F\cap  X}(X)$ and that $F\cap X=\Cl_X{\Int_X(F\cap  X)}$. Let us denote $B'=J_{F\cap X}(X)$. Then
 $\beta(B')=J_{\Cl_{\beta X}(F\cap X)}(\beta X)$. If we prove $\Cl_{\beta X}(F\cap X)=F$, then $\beta(B')=\beta(B)$, so that $B'=B$, and hence $B=J_{F\cap X}( X)$ as claimed.

 Since $F$ is closed in $\beta X$, we have $\Cl_{\beta X}(F\cap X)\subseteq F$. If $F\neq \Cl_{\beta X}(F\cap X)$, pick $x\in F\setminus \Cl_{\beta X}(F\cap X)$. There exists an open neighborhood $U$ of $x$ such that $U\cap (F\cap X)=\emptyset.$ Since $x\in F=\Cl_{\beta X}\Int_{\beta X} F$, the set $U\cap \Int_{\beta X} F$ is nonempty and open in $\beta X$. Density of $X$ in $\beta X$ implies $U\cap \Int_{\beta X} F\cap  X\neq \emptyset$ which is in contradiction with $U\cap F\cap  X=\emptyset$. This shows $F=\Cl_{\beta X}(F\cap X).$ To finish the proof of the forward implication we
apply \Cref{spust zaprtja notranjosti} for $Y=\beta X$ to obtain
$F\cap X=\Cl_X\Int_X(F\cap X)$.

Suppose now that $B=J_F(X)$ and $F$ is the closure of its interior in $X$. If $0\leq f_\alpha\nearrow f$ in $C_b(X)$ and $f_\alpha\in J_F(X)$ for each $\alpha$, then $f_\alpha|_{\Int F}\nearrow f|_{\Int F}$, by \Cref{supst na odprto}. Since $f_\alpha\equiv 0$ on $\Int F$, we have $f\equiv 0$ on $\Int F$, so that by continuity of $f$ we conclude $f\equiv 0$ on $\overline{\Int F}=F.$ Hence, $f\in J_F(X)$ and $B=J_F(X)$ is a band in $C_b(X)$.

(b) If $B$ is a projection band in $C_b(X)$, by \Cref{CharOfBands} there exists a clopen subset $F\subseteq \beta X$ such that
 $\beta(B)=J_F(\beta X)$. By the proof of (a) we know that $B=J_{F\cap X}(X)$. Since $F$ is clopen in $\beta X$, the definition of the relative topology on $X$ yields that $F\cap X$ is clopen in $X$.

 If $B=J_F(X)$ for some clopen subset $F$ of $X$, one can repeat the arguments of the proof of \Cref{bands in c_0}(b) for the appropriately chosen nonnegative function $f\in C_b(X)$ to see that $B$ is a projection band.
%
\end{proof}

The most important ingredient of the proof of \Cref{bands in Tychonoff}(a) is that the set $F$ has an extremely large interior. If we start with a closed ideal $J$ that is not a band in $C_b(X)$, then again we conclude
$\beta(J)=J_F(\beta X)$ for some closed subset $F$ of $\beta X$.
However, the set $F$ can have a very small interior in $\beta X$. In the worst case $F\cap X$ can be empty and in this case we obviously cannot have $J=J_{F\cap X}(X)$ unless $J=C_b(X)$.

The critical step for investigating bands and projection bands in $C_b(X)$ where $X\in T_{3\frac 1 2}$ was the construction of $\beta X$ and the isomorphism between $C_b(X)$ and $C(\beta X)$. As we already mentioned, for a Hausdorff space $X$ there is a way how to construct a space $X/_\sim\in T_{3\frac 1 2}$ such that $C_b(X)\cong C_b(X/_\sim)$ and that the isomorphism in question is induced by a continuous surjective function between $\tau\colon X\to X/_\sim$.
Since in this section $X/_\sim$ and $\tau$ will be needed in details, we recall the construction.

On a Hausdorff space $X$ we define an equivalence relation $\sim$ as follows:
$x\sim y$ \Iff $f(x)=f(y)$ for all continuous function $f\colon X\to \RR$. By $\tau\colon X\to X/_\sim$ we denote the mapping which maps $x$ into its equivalence class $[x]\in X/_\sim$. Then for each $f\in C(X)$ there exists the unique function $\widetilde f\colon X/_\sim\to \mathbb R$ such that the following diagram commutes.
\begin{center}
\begin{tikzpicture}[bij/.style={fill=white,inner sep=2pt}]
\matrix (m) [matrix of math nodes, row sep=3em,
column sep=3em, text height=1.5ex, text depth=0.25ex]
{ X & \RR \\
   & {X/_\sim} \\ };
\path[->]
(m-1-1) edge node[above] {$f$} (m-1-2)
(m-1-1) edge node[below] {$\tau$} (m-2-2);
\path[dashed,->]
(m-2-2) edge node[right] {$ \widetilde f $} (m-1-2);
\end{tikzpicture}
\end{center}
Indeed, we define $\widetilde f$ as $\widetilde f([x]):=f(x)$.
If $[x]=[x']$, then $x\sim x'$, so that $f(x)=f(x')$ as $f$ is continuous. This proves that $\widetilde f$ is well-defined and that $f=\widetilde f\circ \tau.$ Let $C'$ be the family
$\{\widetilde f:\, f\in C(X)\}$. By endowing $X/_\sim$ by the weak topology induced by $C'$, \cite[Theorem 3.7]{Gillman-Jerison:76} implies $X/_\sim\in T_{3 \frac 1 2}$ and that $\tau\colon X\to X/_\sim$ is continuous. If $X\notin T_{3\frac 1 2}$ is functionally Hausdorff, then $\tau$ is just the identity mapping and the topology on $X/_\sim$ is strictly weaker than the topology of $X$.

The mapping $\tau\colon X\to X/_\sim$ induces the isomorphism $\Phi_\tau\colon C(X/_\sim)\to C(X)$ defined by $f\mapsto f\circ \tau.$ By \cite[Theorem 1.6]{Gillman-Jerison:76} $\Phi_\tau$ is also a lattice isomorphism. By applying \cite[Theorem 1.9]{Gillman-Jerison:76}
the restriction of $\Phi_\tau$ to $C_b(X/_\sim)$ induces the isomorphism (which we denote again by $\Phi_\tau$) $\Phi_\tau\colon C_b(X/_\sim)\to C_b(X)$.

The following lemma explains how ideals of the form $J_F(X)$ are transformed by $\Phi_\tau$ and its inverse.

\begin{lemma}\label{fixed ideals transform}
Let $X$ be a Hausdorff space and $X/_\sim$ as before.
Then
\begin{enumerate}
\item [(a)] For  $F\subseteq X$ we have $\Phi_\tau^{-1}(J_F(X))=J_{\overline{\tau(F)}}(X/_\sim).$
\item [(b)] For $F\subseteq X/_\sim$ we have $\Phi_\tau(J_F(X/_\sim))=J_{\tau^{-1}(F)}(X)$.
\end{enumerate}
\end{lemma}

\begin{proof}
(a) Take a function $f\in J_F(X)$. Then $\widetilde f\circ \tau=f$ implies that $\widetilde f\equiv 0$ on $\tau(F)$. By continuity we conclude $\widetilde f\equiv 0$ on $\overline{\tau(F)}$, so that
$\Phi_\tau^{-1}(f)=\widetilde f\in J_{\overline{\tau(F)}}(X/_\sim).$
Conversely, if $\widetilde f\in J_{\overline{\tau(F)}}(X/_\sim)$, then $f=\widetilde f\circ \tau\equiv 0$ on $F$.

(b) If $f\in \Phi_\tau(J_F(X/_\sim))$, then $f=\widetilde f\circ \tau$ for some $\widetilde f\in J_F(X/_\sim)$. Pick $x\in \tau^{-1}(F)$. Then $\tau(x)\in F$ and $f(x)=\widetilde f(\tau(x))=0$; hence $f\in J_{\tau^{-1}(F)}(X)$.  Conversely,  choose  $f\in J_{\tau^{-1}(F)}(X)$. Then $f=\widetilde f\circ \tau=\Phi_\tau(\widetilde f)$. Since $f|_{\tau^{-1}(F)}\equiv 0$ and $\tau(\tau^{-1}(F))=F$ we have
$\widetilde f\equiv 0$ on $F$, and therefore  $f\in \Phi_\tau(J_F(X/_\sim))$.
\end{proof}

\begin{thm}\label{proj band in hausdorff}
Let $X$ be a Hausdorff space. A closed ideal $J\subseteq C_b(X)$ is a projection band \Iff $J=J_F(X)$ for some clopen set $F\subseteq X$.
\end{thm}

\begin{proof}
Obviously, $J$ is a projection band in $C_b(X)$ \Iff $\Phi_\tau^{-1}(J)$ is a projection band in $C_b(X/_\sim)$. Therefore, by \Cref{bands in Tychonoff} $J$ is a projection band in $C_b(X)$ \Iff there exists a clopen set $F\subseteq X/_\sim$ such that $\Phi_\tau^{-1}(J)=J_F(X/_\sim)$. Therefore $J$ is a projection band \Iff it is of the form
$$J=\Phi_\tau(J_F(X/_\sim))=J_{\tau^{-1}(F)}(X)$$ for some clopen set $F$ in $X/_\sim$.  Since $\tau$ is continuous, $\tau^{-1}(F)$ is clopen and the proof is finished.
\end{proof}

The natural question that arises here is whether bands in $C_b(X)$ where $X$ is Hausdorff are of the form $J_F(X)$ where $F=\overline{\Int F}.$ In general, the answer is no.

\begin{ex}\label{counterex}
Let $\tau_e$ be the Euclidean topology on $\RR$.
Let $X=\RR$ be equipped with topology $\tau$ generated by ${\mathcal B}=\tau_e\cup \{U\setminus N:\; U\in\tau_e\}$,
where $N=\bigcup_{n=1}^\infty [\tfrac 1{2n},\tfrac 1{2n-1}]$. The topology $\tau$ is stronger than $\tau_e$ and since $\RR$ is a functionally Hausdorff space, $X$ is functionally Hausdorff as well.
Therefore, we have $X/_\sim=\RR$ as sets.

\begin{itemize}
  \item \emph{The topology on $X/_\sim$ is stronger than the Euclidean topology.}
  \end{itemize}
  Indeed, let $F\subseteq \RR$
be closed (in the Euclidean topology). Then for every $x\not\in F$ there
exists a continuous function $\varphi\colon\RR\to\RR$ such that $\varphi(x)=1$ and $\varphi(F)=0$.
Since $\varphi\colon X\to\RR$ is continuous and
$x\in\{y\in X:\; |\varphi(x)-\varphi(y)|<1\}\subseteq X\setminus F$,
$X\setminus F$ is open, and hence $F$ is a closed subset of $X/_\sim$.
\begin{itemize}
  \item \emph{The space $X$ is not $T_3$.}
  \end{itemize}
 Indeed, by definition the set $N$ is closed in $X$ and $0\not\in N$ .
Let $U,V\subseteq X$ be open sets such that $0\in U$ and  $F\subseteq V$. For every $n\in\NN$, $\tfrac 1{2n}\in V$
hence $(\tfrac 1{2n}-r_n,\tfrac 1{2n}+r_n)\subseteq V$ for some $r_n>0$. So there exists $x_n\in (\tfrac 1{2n+1},\tfrac1 {2n})\cap V$.
The sequence $(x_n)$ lies in $\RR\setminus N$ and $\inf x_n=0$, hence $\emptyset \ne U\cap\{x_n:\; n\in\NN\}\subset U\cap V$.
Therefore we can not separate the point $0$ from the closed set $N$. We actually proved $0\in \Cl_{X/_\sim} N$ and similarly, one can prove $0\in \Cl_{X/_\sim}(\bigcup_{n=1}^\infty (\tfrac 1{2n},\tfrac 1{2n-1}))$.

Because $\RR\setminus \{0\}$ and $X\setminus \{0\}$ are homeomorphic they are also homeomorphic to $X/_\sim\setminus \{0\}$.
Hence $\Cl_{\RR} A$, $\Cl_{X} A$ and $\Cl_{X/_\sim} A$ may differ only for the point $0$. The same is true for interiors of the given set.

\begin{itemize}
  \item \emph{$\Cl_{X/_\sim}\Int_{X/_\sim} F=F$ where $F=N\cup\{0\}.$}
  \end{itemize}
 Since $F$ is closed in $\RR$ and since  topologies on $X$ and $X/_\sim$ are stronger than the Euclidean topology,
$F$ is also closed in both $X$ and $X/_\sim$. Therefore we have
$\Cl_{X/_\sim}\Int_{X/_\sim} F \subseteq F.$ Due to the above remark and the fact that $\Int_\RR F=\bigcup_{n=1}^\infty (\tfrac 1{2n},\tfrac 1{2n-1})$,
we have $\Int_X F=\Int_{X/_\sim}F=\bigcup_{n=1}^\infty (\tfrac 1{2n},\tfrac 1{2n-1})$ from where it follows that
$$F\setminus\{0\}=N=\bigcup_{n=1}^\infty [\tfrac 1{2n},\tfrac 1{2n-1}]\subseteq \Cl_{X/_\sim}\Int_{X/_\sim} F.$$

Because $\RR\setminus N$ is open in $X$ and $(\RR\setminus N)\cap \Int_X F=\emptyset$, we have $0\notin \Cl_X\Int_X F$, and hence
$$\Cl_X\Int_X F=F\setminus \{0\}\subseteq \Cl_{X/_\sim}\Int_{X/_\sim} F\subseteq \Cl_\RR\Int_\RR F=F.$$
Since $0\in \Cl_{X/_\sim} \Int_{X_\sim} F$, we finally conclude $\Cl_{X/_\sim}\Int_{X/_\sim} F=F$.

\begin{itemize}
  \item \emph{$J_F(X)$ is a band in $C_b(X/_\sim)$; yet $\Cl_X \Int_XF=F\setminus \{0\}.$ }
  \end{itemize}
The set $F$ is closed in $X$ and $X/_\sim$. The mapping $\tau$ is the identity mapping as $X$ is functionally Hausdorff. Since the closed ideal $J_F(X)$ satisfies $$\Phi_\tau^{-1}(J_F(X))=J_{\tau(F)}(X/_\sim)=J_F(X/_\sim)$$
and since $\Cl_{X/_\sim}\Int_{X/_\sim} F=F$, we conclude that $J_F(X/_\sim)$ is a band in $C_b(X/_\sim)$. Therefore, $J_F(X)$ is band in $C_b(X)$, yet
$\Cl_X\Int_XF=F\setminus \{0\}.$
\end{ex}

The topological space constructed in \Cref{counterex} shows us that there are examples of spaces which satisfy weaker separation axioms than $T_{3\frac 1 2}$ and at the same time conclusions of \Cref{supst na odprto} still hold.

\begin{ex}\label{Spust na TFC}
Let $\tau_e$ be the Euclidean topology on $\RR$.
Let $X=\RR$ be equipped with topology $\tau$ generated by ${\mathcal B}=\tau_e\cup \{U\setminus N:\; U\in\tau_e\}$,
where $N=\bigcup_{n=1}^\infty [\tfrac 1{2n},\tfrac 1{2n-1}]$.
By \Cref{counterex} we already know that $X$ is functionally Hausdorff while it is not $T_{3\frac 1 2}$. Take an arbitrary open set $A$ in $X$ and assume $0\leq f_\alpha\nearrow f$ in $C(X)$. We claim that $f_\alpha|_A\nearrow f|_A$ in $C(A)$.
Suppose there is $g:A\to \RR$ such that $f_\alpha|_A\leq g$ and $g<f|_A.$

Suppose first that there is $0\neq x\in A$ with $g(x)<f(x)$. Then there exists $r>0$ such that $(x-r,x+r)\subseteq A$ and $g(y)<f(y)$ for each $y\in (x-r,x+r)$. Take any function $\varphi:\RR\to [0,1]$ which is continuous with respect to the Euclidean topology $\tau_e$ and satisfies $\varphi(x)=1$ and $\varphi\equiv 0$ on $\mathbb R\setminus (x-r,x+r)$. Since $\tau$ is stronger than $\tau_e$, the function $\varphi:X\to [0,1]$ is also continuous. Then $g\leq f|_A-\varphi|_A(f|_A-g)$
and $f_\alpha\leq f-\varphi(f-g)<f$ for each $\alpha$
where $\varphi(f-g)$ is defined as
$$\varphi(f-g)(x)=\left\{\begin{array}{ccl}
 \varphi(x)(f(x)-g(x))& : &x\in A\\
 0& : &x\in  X\setminus [x-\tfrac{r}{2},x+\tfrac{r}{2}]
 \end{array}\right..$$ This is in contradiction with $f_\alpha\nearrow f$.

If $0\in A$ and $g(0)<f(0)$, then $g(y)<f(y)$ on some open neighborhood $U$ of $0\in X$, so that  $g(y)<f(y)$ on some open neighborhood of $y_0\in A\setminus \{0\}$ as well. By the previous paragraph this is impossible, and hence $f_\alpha|_A\nearrow  f|_A.$
\end{ex}

\begin{quest}
Is there an example of a functionally Hausdorff space $X$, an open set $A$ in $X$ and a net $(f_\alpha)$ in $C(X)$ such that $f_\alpha\nearrow f$ in $C(X)$ but $f_\alpha|_A\nearrow  f|_A$ does not hold in $C(A)$?
\end{quest}

If $X\in T_{3\frac 1 2}$, the mapping $F\mapsto J_F(X)$ is a one-to-one mapping on the set of all closed subsets of $X$.
If $X$ is merely a Hausdorff space, then it is possible that different closed sets induce the same fixed ideal. However, for a set $F\subseteq X$ there is the largest closed set $F'$ such that $J_F(X)=J_{F'}(X)$.

\begin{thm}\label{fixed ideals with F}
Let $X$ be a Hausdorff space and $F$ a subset of $X$. Then $\tau^{-1}(\overline{\tau(F)})$ is the largest closed set $F'$ with the property $J_F(X)=J_{F'}(X).$ Furthermore, we have
$$J_F(X)=J_{\overline{\tau^{-1}(\tau(F))}}(X)=J_{\tau^{-1}(\overline{\tau(F)})}(X).$$
\end{thm}

\begin{proof}
%
From $J_{\tau(F)}(X/_\sim)=J_{\overline{\tau(F)}}(X/_\sim)$ and  \Cref{fixed ideals transform} we conclude
$$J_{\overline{\tau^{-1}(\tau(F))}}(X)=J_{\tau^{-1}(\tau(F))}(X)=
J_{\tau^{-1}(\overline{\tau(F)})}(X).$$
We claim $J_F(X)=J_{\tau^{-1}(\tau(F))}(X)$. Note first that $J_{\tau^{-1}(\tau(F))}(X)\subseteq J_F(X)$ follows from the fact $F\subseteq \tau^{-1}(\tau(F)).$ For the opposite inclusion, take $f\in J_F(X)$ and $x\in \tau^{-1}(\tau(F))$. Then $\tau(x)\in \tau(F)$, so that there is $y\in F$ with $[x]=[y]$. By  definition of $\sim$ we have $f(x)=f(y)=0$. This gives $f\equiv 0$ on $\tau^{-1}(\tau(F))$ which proves the claim.

The only thing that remains to be proved is that $\tau^{-1}(\overline{\tau(F)})$ is the largest closed set $G$ with the property $J_F(X)=J_G(X)$. In order to prove this, let $G$ be an arbitrary closed set with $J_G(X)=J_F(X)$.
Then $J_{\tau(G)}(X/_\sim)=J_{\tau(F)}(X/_\sim)$.
Pick $x\in G$ and assume $\tau(x)\notin \overline{\tau(F)}$. Since $X/_\sim\in T_{3\frac 1 2}$, there is $\widetilde f\in C_b(X/_\sim)$ such that
$\widetilde f(\tau(x))=1$ and $\widetilde f\equiv 0$ on $\overline{\tau(F)}.$  By \Cref{fixed ideals transform} we have $f=\Phi_\tau(\widetilde f)\in J_{\tau^{-1}(\overline{\tau(F)})}(X)=J_F(X)=J_G(X)$. Since $x\in G$ we have $f(x)=0$ which is a contradiction.
\end{proof}


\begin{cor}
A Hausdorff space $X$ is $T_{3\frac 1 2}$ \Iff $J_F(X)=J_G(X)$ for closed subsets $F$ and $G$ implies $F=G$.
\end{cor}

\begin{proof}
If $X\in T_{3 \frac 1 2}$, then the conclusion follows from the second paragraph of \Cref{Section: order density}.

Suppose now that $X\notin T_{3 \frac 1 2}$. If $X$ is functionally Hausdorff, then $X/_\sim=X$ as sets and $\tau$ is the identity mapping. Since topology of $X/_\sim$ is strictly weaker than the topology of $X$, $\tau$ is not a closed mapping. Hence, there exists a closed set $F$ in $X$ which is not closed in $X/_\sim$. Since $\tau(F)$ is not closed in $X/_\sim$, $F$ is a proper subset of $\tau^{-1}(\overline{\tau(F)})$. By the assumption we have $J_F(X)\neq J_{\tau^{-1}(\overline{\tau(F)})}(X)$ which contradicts \Cref{fixed ideals with F}.

If $X$ is not functionally Hausdorff,
there are two different points $x$ and $x'$ which cannot be separated by continuous functions. Then $J_{\{x\}}(X)=J_{\{x'\}}(X)$. On the other hand, the assumption on the uniqueness  implies $x=x'$. Again a contradiction.
\end{proof}

We already mentioned that order complete vector lattices posses the projection property. It is well-known that the converse statement does not hold. Indeed, the set $E$ of all real bounded functions on $[0,1]$ assuming only finitely many different values becomes a vector lattice when ordered pointwise. This vector lattice is not order complete, yet it has the projection property
(see e.g. \cite[Exercise 12.6.(ii)]{Zaanen:96}). Order complete vector lattices are also uniformly complete (see e.g. \cite{Aliprantis:03}). Although neither projection property nor uniform completeness imply order completeness of a given vector lattice, together they do \cite[Theorem 1.59]{Aliprantis:03}. Since by \cite[Proposition 1.1.8]{Meyer-Nieberg:91} Banach lattices are uniformly complete, a given Banach lattice has the projection property \Iff it is order complete.
If we assume that in a given Banach lattice every closed ideal is a band, then the norm on is order continuous. Actually, there is an equivalence between the last two statements
(see e.g. \cite[Theorem 17.17]{Zaanen:96}).
In the following result we determine when $C_b(X)$ is order continuous.
It turns out that whenever $X$ is Hausdorff then $C_b(X)$ is order continuous \Iff $C_b(X)$ is finite-dimensional.

\begin{cor}\label{closed ideals bands functionally hasudorff}
Let $X$ be a Hausdorff space with the property that every closed ideal in $C_b(X)$ is a band. Then the following assertions hold.
\begin{enumerate}
\item [(a)] $X$ is functionally Hausdorff \Iff $X$ is finite.
\item [(b)] $C_b(X)$ is finite-dimensional.
\end{enumerate}
\end{cor}

\begin{proof}
(a) By the assumption and \Cref{bands in Tychonoff} every maximal ideal in $C_b(X/_\sim)$ is of the form $J_F(X/_\sim)$ for some closed subset $F$ of $X/_\sim$. Hence, $X/_\sim$ is compact by \Cref{maximals compact}. Also,
 for each $x\in X/_\sim$ the closed ideal $J_{\{x\}}(X/_\sim)$ is a band, so that $\{x\}$ is open in $X/_\sim$, again by \Cref{bands in Tychonoff}. Thus,  $X/_\sim$ is discrete and since it is also compact, it is finite.
 Since $X=X/_\sim$ as sets, $X$ is finite. The converse statement obviously holds.

(b) By (a) we conclude $X/_\sim$ is finite and since $C_b(X)\cong C_b(X/_\sim)$ it follows that $C_b(X)$ is finite-dimensional.
\end{proof}

\section{Extremally disconnected spaces and order completeness}\label{extr oc}

Although the function spaces $L_p(\mu)$ are order complete whenever $0<p<\infty$, the Banach lattice of continuous functions on a compact Hausdorff space is very rarely order complete. This happens only in the case when the underlying topological space is \term{extremally disconnected}, i.e., the closure of every open set is again open. In regular extremally disconnected spaces every point has a basis consisting of clopen sets.

In this section, as an application of \Cref{bands in c_0} and  \Cref{bands in Tychonoff} among locally compact Hausdorff spaces and $T_{3\frac 1 2}$-spaces, respectively, we characterize order complete $C_0(X)$ and $C_b(X)$ spaces, respectively. The characterization is the same as in the case of compact Hausdorff spaces. Before we proceed to results, we recall the following characterization from \cite[Proposition 2.1.4]{Meyer-Nieberg:91}.

\begin{thm}\label{OC C(X)}
For a compact Hausdorff space $X$ the following statements are equivalent.
\begin{enumerate}
  \item [(a)] $C(X)$ is order complete.
  \item [(b)] $C(X)$ has the projection property.
  \item [(c)] $X$ is extremally disconnected.
\end{enumerate}
\end{thm}

The following theorem is the analog of \Cref{OC C(X)} for $C_0(X)$. Among the already mentioned \Cref{bands in c_0} and \Cref{bands in Tychonoff} we also require the fact that order complete vector lattices are precisely those that are uniformly complete and have the projection property.

\begin{thm}\label{OC C_0(X)}
For a locally compact Hausdorff space $X$ the following statements are equivalent.
\begin{enumerate}
\item [(a)] $C_0(X)$ is order complete.
\item [(b)] $C_0(X)$ has the projection property.
\item [(c)] $X$ is extremally disconnected.
\end{enumerate}
\end{thm}

\begin{proof}
(a) $\Rightarrow$ (b) is obvious and (b) $\Rightarrow$ (a) follows from the fact that $C_0(X)$ is a Banach lattice.

(b) $\Rightarrow$ (c) Let $U$ be an arbitrary open set in $X$ and let us denote by $F$ the closure of $U$. Then
$J_F^0(X)$ is a band in $C_0(X)$ by \Cref{bands in c_0}. By the assumption $J_F^0(X)$ is a projection band, so that by \Cref{bands in c_0} there exists a clopen subset $G$ of $X$ such that $J_F^0(X)=J_G^0(X)$. As an application of Urysohn's lemma for locally compact Hausdorff spaces, one can prove $F=G$. Hence, $F$ is open and $X$ is extremally disconnected.

(c) $\Rightarrow$ (b)
Suppose now that $X$ is extremally disconnected and take an arbitrary band $J$ in $C_0(X)$. Then there exists a closed set $F$ which is the closure of its interior and $J=J_F^0(X)$. Since $X$ is extremally disconnected, $F$ is clopen, so that $J_F^0(X)$ is a projection band by \Cref{bands in c_0}. Therefore, $C_0(X)$ has the projection property. \end{proof}

Subspaces of extremally disconnected spaces are not necessary extremally disconnected as the following example shows.

\begin{ex}\label{counterexample extremally disconnected}
Let $(X,\tau)$ be any topological space that is not extremally disconnected and $\infty$ an object that is not in $X$.
Define $Y=X\cup \{\infty\}$ and
$$\tau_1=\{U\cup\{\infty\}:\; U\in \tau\}\cup\{\emptyset\}.$$
It is easy to see that $\tau_1$ is a topology on $Y$, and that closed sets in $Y$ are precisely the closed sets of $X$ and whole $Y$.
It should be noted that the space $(Y,\tau_1)\in T_0$ while $(Y,\tau_1)\notin T_1.$
It is also clear that the $(X,\tau)\hookrightarrow (Y,\tau_1)$ is continuous and that $X$ is not dense in $Y$.

We claim that $Y$ is extremally disconnected. Indeed, pick an open set $U$ in $Y$. If $U=\emptyset$, then its closure is still empty which is open in $Y$. Otherwise $U=V\cup\{\infty\}$ for some open set $V$ in $X$. Since $\infty\in U$, the closure of $U$ is $Y$ which is open in $Y$.
This proves the claim.
\end{ex}

Although \Cref{counterexample extremally disconnected} shows that a subspace of an extremally disconnected space is not necessary extremally disconnected, there are positive results \cite{Arhangelski:74}. Due to the importance of the following proposition we include its short proof for the sake of completeness.

\begin{prop}\label{on ex disc sp}
Let $X$ be a subspace of an extremally disconnected space $Y$.
Then $X$ is extremally disconnected in either of the following cases.
\begin{enumerate}
\item [(a)] $X$ is open in $Y$.
\item [(b)] $X$ is dense in $Y$.
\end{enumerate}
\end{prop}

\begin{proof}
(a) Take an arbitrary open subset $U\subseteq  X$. Then $U$ is also open in $Y$, so that $\Cl_Y(U)$ is open in $Y.$ Then by the definition of the relative topology we have $\Cl_{X}(U)=\Cl_Y(U)\cap  X$; hence the closure of $U$ in $X$ is open.

(b) Take any open set $U$ in $X$. Then there exists an open set $V$ in $Y$ such that $U=V\cap X$. The set $\Cl_Y V$ is open in $Y$. If
$\Cl_X(V\cap X)=\Cl_Y V\cap X$, then $\Cl_X U=\Cl_X(V\cap X)$ is open in $X$ and $X$ is extremally disconnected.

To prove $\Cl_X(V\cap X)=\Cl_Y V\cap X$  observe that $\Cl_X (V\cap X)\subseteq \Cl_Y V\cap X$  follows from the fact that the set $\Cl_Y V\cap X$  is closed in $X$ and contains $V\cap X.$ For the opposite inclusion, pick $x\in \Cl_Y V\cap X$ and any open neighborhood $W$ of $x$ in $X$. Then there exists an open set $W'$ in $Y$ such that $W=W'\cap X$. Since $x\in \Cl_YV$, we have  $W'\cap V\neq \emptyset$, and since $X$ dense in $Y$ the set $W\cap (V\cap X)=W'\cap V\cap X$ is nonempty as well. This proves $x\in \Cl_X(V\cap X)$.
\end{proof}

We apply \Cref{on ex disc sp} in the following two cases: 
\begin{itemize}
  \item If $X^+$ is extremally disconnected, the dense subspace $X$ of $X^+$ is extremally disconnected. 
  \item If $X\in T_{3\frac 1 2}$ and $\beta X$ is extremally disconnected, the dense subspace $X$ of $\beta X$ is extremally disconnected. 
\end{itemize}


The following result now follows immediately from \Cref{OC C(X)} and \Cref{OC C_0(X)}.

\begin{cor}
Let $X$ be a locally compact Hausdorff space. If  $C(X^+)$ is order complete, then $C_0(X)$ is order complete.
\end{cor}

The preceding corollary can be proved directly without involving extremally disconnected spaces and their connection to order completeness. This follows immediately from $C_0(X)\cong J_{\{\infty\}}(X^+)$ and from the fact that every order ideal in an order complete vector lattice is order complete on its own (see e.g. \cite[Theorem 1.62]{Aliprantis:03}).

The following diagram shows the connection between order completeness and extreme disconnectedness for the locally compact Hausdorff space and its one-point compactification.
$$\begin{array}{ccc}
C(X^+)\textrm{ is order complete } &\Leftrightarrow & X^+ \textrm{ is extremally disconnected }\\
 \Downarrow && \Downarrow \\
 C_0(X) \textrm{ is order complete} & \Leftrightarrow &  X \textrm{ is  extremally disconnected}
\end{array}$$

None of the arrows above cannot be reversed as the following example shows.

\begin{ex}
Take $X=\mathbb N$. If we endow $X$ with a discrete topology, it becomes a locally compact noncompact Hausdorff space. The space $C_0(X)$ is precisely the order complete Banach lattice $c_0$, however $C(X^+)$ is isometrically lattice isomorphic to the Banach lattice $c$ of all convergent sequences which is not order complete.
\end{ex}

We continue this section with the variant of \Cref{OC C_0(X)} for the Banach lattice $C_b(X)$ where $X\in T_{3 \frac 1 2}.$ Since the proof is basically the same as the proof of \Cref{OC C_0(X)} we omit it. We should mention that instead of \Cref{bands in c_0} one should use \Cref{bands in Tychonoff}.

\begin{thm}\label{OC C_b(X)}
For $X\in T_{3\frac 1 2}$ the following statements are equivalent.
\begin{enumerate}
\item [(a)] $C_b(X)$ is order complete.
\item [(b)] $C_b(X)$ has the projection property.
\item [(c)] $X$ is extremally disconnected.
\end{enumerate}
\end{thm}

The following diagram reveals the connection between order completeness and extreme disconnectedness for $X\in T_{3\frac 1 2}$ and its Stone-\v Cech compactification $\beta X$. Since the algebras $C(\beta X)$ and $C_b(X)$ are isometrically lattice and algebra isomorphic all the arrows here are also reversible in contrast with the arrows in the diagram involving $C_0(X)$ and $C(X^+)$.
$$\begin{array}{ccc}
C(\beta X)\textrm{ is order complete } &\Leftrightarrow & \beta X \textrm{ is extremally disconnected }\\
 \Updownarrow && \Updownarrow \\
 C_b(X) \textrm{ is order complete} & \Leftrightarrow &  X \textrm{ is  extremally disconnected}
\end{array}$$




If $X\in T_{3\frac 1 2}$ then $X$ is extremally disconnected \Iff
$\beta X$ is. The following example shows that such equivalence does not hold for a locally compact space and its one-point compactification.

\begin{ex}\label{exstremni primer za X+}
The one-point compactification $\mathbb N^+$ of $\mathbb N$ is homeomorphic to the totally disconnected compact space $X=\{\frac{1}{n}:\; n\in \mathbb N\}\cup\{0\}$ which is not extremally disconnected. To see this, let us denote  sets $\{\frac{1}{2n}:\; n\in \mathbb N\}$ and $\{\frac{1}{2n-1}:\; n\in \mathbb N\}$ by $A$ and $B$, respectively. Then both $A$ and $B$ are open and neither of them is closed. The closure $\overline{A}$ of $A$ is $A\cup \{0\}=X\setminus B$ which is not open.
\end{ex}

In \Cref{ex disc -> compact} we prove that $X^+$ is never extremally disconnected when $X$ is a $\sigma$-compact locally compact noncompact Hausdorff space. Before we state and prove the aforementioned result we need the following lemma.

\begin{lemma}
A locally compact Hausdorff space $X$ is $\sigma$-compact \Iff the point $\infty$ in $X^+$ has a countable basis of open sets.
\end{lemma}

\begin{proof}
Suppose first that $X$ is $\sigma$-compact. Then by
\cite[Theorem XI.7.2]{Dugundji:66} there exist relatively compact open sets $U_n$ in $X$ with $X=\bigcup_{n=1}^\infty U_n$ and $\overline{U_n}\subseteq U_{n+1}$ for each $n\in \mathbb N.$ If $V$ is an open neighborhood of $\infty$, then $X\setminus V$ is compact in $X$.
Since the countable family of open sets $(U_n)$ is increasing and covers $X$ there exists $n\in \mathbb N$ such that $X\setminus V\subseteq \overline{U_n}.$ Then $V_n:=X^+\setminus \overline{U_n}\subseteq V$, and hence the family $(V_n)$ is a countable basis for the point $\infty$ in $X^+$.

For the converse let $(V_n)$ be a countable basis for $\infty$ in $X^+$. Obviously we can assume that $V_n\subseteq V_m$ whenever $m\leq n$. The sets $F_n:=X\setminus V_n$ are compact in $X$ and $F_m\subseteq F_n$ whenever $m\leq n.$ Since $X^+$ is Hausdorff,
$\bigcap_{n=1}^\infty V_n=\{\infty\}$, so that
$$\bigcup_{n=1}^\infty F_n=\bigcup_{n=1}^\infty(X\setminus V_n)=X \setminus  \bigcap_{n=1}^\infty V_n=X.$$
This proves that $X$ is $\sigma$-compact.
\end{proof}

\begin{thm}\label{ex disc -> compact}
Let $X$ be a $\sigma$-compact locally compact Hausdorff space. If $X^+$ is extremally disconnected, then $X$ is compact.
\end{thm}

\begin{proof}
Suppose $X$ is not compact. Since $X$ is $\sigma$-compact, there exist relatively compact open sets $(U_n)$ in $X$ that cover $X$ and $\overline{U_n}\subseteq U_{n+1}$.  Then the family $(X^+\setminus \overline{U_n})$ is a countable basis of $\infty$. For each $n\in\mathbb N$ we choose $x_n\in U_n\setminus \overline{U_{n-1}}.$ Since each basis neighborhood $X^+\setminus \overline{U_n}$ contains all except finitely many terms of the sequence $(x_n)$ we have $x_n\to \infty$ in $X^+$.

Since $X^+$ is Hausdorff no subnet
of $(x_n)_{n\geq k}$ converges in $X$ for each $k\in \mathbb N$. This implies that for each $k\in\mathbb N$ the set $F_k:=\{x_n:\, n\geq k\}$ is closed in $X^+$ and so in $X$.

We claim that for each $n\in\mathbb N$ there exists a clopen neighborhood  $W_n$ of $x_n$ in $X$ such that $W_n\cap W_m=\emptyset$ whenever $n\neq m$
and $W_n\cap F_m=\emptyset$ whenever $n<m.$ We construct the desired sets inductively. Since the set $X\setminus F_2$ is open and $x_1\in X\setminus F_2$, there is a clopen neighborhood $W_1$ of $x_1$ that is contained in $X\setminus F_2.$ Obviously $W_1\cap F_m=\emptyset$ for $m>1.$ Suppose now that the sets $W_1,\ldots,W_n$ with the required properties are already constructed.
Since the set $(X\setminus W_1\cup\cdots\cup W_n)\setminus F_{n+2}$ is an open neighborhood of $x_{n+1}$ in $X$, it contains a clopen neighborhood $W_{n+1}$ of $x_{n+1}$ in $X$. By construction $W_{n+1}$ does not contain $x_k$ for $k\neq n+1.$
%

The set $W=\bigcup_{n=1}^\infty W_{2n}$ is open in $X$. Since $X$ is locally compact, $X$ is open in $X^+$, so that $W$ is open in $X^+$ as well.
Because $x_{2n}\to\infty$ we have $\infty \in\overline W$.
From $x_{2n-1}\to \infty$ we conclude that every neighborhood $U$ of $\infty$ contains infinitely many
elements $x_{2n-1}$. On the other hand $W_{2n-1}\cap W=\emptyset$ implies  $x_{2n-1}\not\in \overline W$ for each $n\in\mathbb N$.
Therefore, for an arbitrary neighborhood $U$ of $\infty$ we have
$U\not\subseteq \overline W$, so that $\overline W$ is not open in $X^+$.
This is in contradiction that $X^+$ is extremally disconnected.
Hence, $X$ is compact.
\end{proof}

\begin{cor}
Let $X$ be a locally compact $\sigma$-compact Hausdorff space. If
$C(X^+)$ is order complete, then $X$ is compact and $C(X)$ is a projection band of codimension one in $C(X^+)$.
\end{cor}

\begin{proof}
If $C(X^+)$ is order complete then $X^+$ is extremally disconnected by \Cref{CharOfBands}, so that by \Cref{ex disc -> compact} $X$ is compact.
Therefore $C(X)=C_0(X)$ and $\infty$ is isolated in $X^+$. This means that $\{\infty\}$ is clopen in $X^+$ and \Cref{CharOfBands} implies that $C(X)=J_{\{\infty\}}(X^+)$ is a projection band in $C(X^+)$. That the codimension of $C(X)$ in $C(X^+)$ is one is obvious.
\end{proof}

There is no hope that \Cref{ex disc -> compact} can be applied to $X$ and its Stone-\v Cech compactification $\beta X$. The reason behind it is that the space $\beta X$ is just too big and the sequential nature of the construction in the proof of \Cref{ex disc -> compact} simply fails.

\begin{ex}
The set of all positive integers $\mathbb N$ endowed with a discrete topology is a $\sigma$-compact locally compact discrete space.
As a discrete space it is definitely extremally disconnected. Since $\mathbb N\in T_{3\frac 1 2}$,  $\beta \mathbb N$ is extremally disconnected; yet $\mathbb N$ is not compact.
\end{ex}

%

\section{Lifting un-convergence}\label{lift un}

In this section we apply our results to the so-called problem of ``\emph{lifting un-convergence}". We first recall some basic facts needed throughout this section.

If $E$ is a Banach lattice, then a net $(x_\alpha)$ is said to \term{un-converge} to a vector $x\in E$ whenever for each $y\in E_+$ we have $|x_\alpha-x|\wedge y\to 0$ in norm. We write $x_\alpha\goesun x$ whenever the net $(x_\alpha)$ un-converges to $x$. This mode of convergence was introduced by V.G. Troitsky \cite{Troitsky:04} under the name of \term{$d$-convergence}. In \cite[Example 20]{Troitsky:04} it was proved that un-convergence in $C_0(X)$ coincides with the uniform convergence on compacta of $X$ whenever $X$ is normal. In particular, when $X$ is a compact Hausdorff space, then un-convergence coincides with the uniform convergence. In the case of $L_p(\mu)$-spaces with $\mu$ finite he proved that un-convergence coincides with convergence in measure
\cite[Example 23]{Troitsky:04}. It is a standard fact from measure theory that a sequence $(f_n)$ converging in measure to $f$ always has a subsequence converging to $f$ almost everywhere (see e.g. \cite[Theorem 2.30]{Folland:99}). In \cite[Proposition 3.1]{GTX} the authors proved that a sequence $(f_n)$ in $L_0(\mu)$ converges almost everywhere to $f\in L_0(\mu)$ whenever $|f_n-f|\wedge g\to 0$ in order of $L_0(\mu)$.
The latter mode of convergence goes back to \cite{Nakano:48, DeMarr:64, Kaplan:97}. Kaplan referred to this convergence as \term{unbounded order convergence} or uo-convergence for short. Until very recently unbounded order convergence was not studied actively and was left out from the active area of research. The systematic study of this mode of convergence and its properties started with papers of Gao, Troitsky and Xanthos \cite{Gao:14, GaoX:14, GTX} and others. The systematic study  and properties of un-convergence started in \cite{DOT}. Among other things, authors proved that un-convergence is topological. On the other hand, uo-convergence is not (see e.g. \cite{Ordman:66}). In \cite{KMT} the authors initiated the study of un-topology, i.e., the topology given by un-convergence. They posed the following question.

\begin{quest}\label{question}
 Let $B$ be a band in a Banach lattice $E$. Suppose that every net in $B$
which is un-null in $B$ is also un-null in $E$. Does this imply that $B$ is a projection band?
\end{quest}

When the norm of $E$ is order continuous, every band is a projection band.  In this case it was proved in \cite{KMT}  that a net $(x_\alpha)$ which is un-null in a sublattice of $E$ it is also un-null in $E$. \cite[Example 4.2]{KMT} shows that there exists a band $B$ in $C[-1,1]$ and a un-null net $(x_\alpha)$ in $B$ which is not un-null in $C[-1,1]$.

The following theorem provides a positive answer to \Cref{question}
for a Banach lattice $C_b(X)$ of all bounded continuous functions on a Hausdorff space.

\begin{thm}\label{bands vs projection bands}
Let $X$ be a Hausdorff space and $J$ a closed ideal in $C_b(X)$. Suppose that every un-null net $(f_\alpha)$ in $J$ is also un-null in $C_b(X)$. Then $J$ is a projection band in $C_b(X)$.
\end{thm}

\begin{proof}
We first consider the case when $X$ is compact.
Since $J$ is a closed ideal in $C(X)$, by \Cref{CharOfClosId} there exists a closed set $F$ such that $J=J_F(X).$ If $F$ is not open, there exists $x\in F\setminus \Int F.$ Let $\mathcal B_x=\{W_\lambda\}_{\lambda\in\Lambda}$
be the set of all open neighborhoods of the point $x$. Since $ X$ is Hausdorff, \cite[VII.I.2]{Dugundji:66} implies
$$\{x\}=\bigcap_{\lambda\in\Lambda}W_\lambda.$$ If $\Lambda$ is finite, $\{x\}$ is open, so that $x\in \Int F$ which is a contradiction. Therefore $\Lambda$ is infinite.

Since $x\in F\setminus \Int F$, for each $\lambda \in \Lambda$ the set $W_\lambda\setminus F$ is nonempty. Pick any $x_\lambda\in W_\lambda\setminus F.$ Also, let $V_\lambda$ be an arbitrary open neighborhood of $x_\lambda$ in $(X\setminus F)\cap W_\lambda.$ By Urysohn's lemma for each $\lambda\in\Lambda$ there exists a nonnegative continuous function $f_\lambda$ such that
$f_\lambda(x_\lambda)=1$ and $f|_{ X\setminus V_\lambda}\equiv 0.$ We define an ordering on the set $\Lambda$ by $\lambda \leq \mu$ \Iff $W_\mu\leq W_\lambda$.
With this ordering $\Lambda$ becomes a directed set.

We claim that $f_\lambda\goesun 0$ in $J$. Pick an arbitrary function $g\in J$. Then $g(x)=0$. Continuity of $g$ implies that for every $\epsilon>0$ there exists an open neighborhood $U$ of $x$ such that $|g(y)|<\epsilon$ for each $y\in U.$  There exists $\lambda_0\in \Lambda$ such that $W_{\lambda_0}\subseteq U$. If $\lambda\geq \lambda_0$, then $V_\lambda\subseteq W_\lambda\subseteq W_{\lambda_0}\subseteq U$. If $y\in  X\setminus V_\lambda$, then $f_\lambda(y)=0$ and if $y\in V_\lambda$, then $|g(y)|<\epsilon.$ We conclude $f_\lambda \wedge |g|<\epsilon$ for all $\lambda\geq \lambda_0$. Therefore, $f_\lambda\goesun 0$ in $J$.

Now we claim that the net $(x_\lambda)$ converges to $x$. Indeed, let $W$ be an arbitrary neighborhood of $x$. Then there exists $\lambda_0$ such that $W_{\lambda_0}\subseteq W$. If $\lambda\geq \lambda_0$, then $x_\lambda\in W_\lambda\subseteq W_{\lambda_0}\subseteq W$ which proves the claim.

The assumption $f_\lambda\goesun 0$ in $J$ implies that $f_\lambda\to 0$ in $C(X)$ which contradicts the fact that $f_\lambda(x_\lambda)=1$ for each $\lambda\in \Lambda.$ Therefore  $F=\Int F$ is clopen; hence $J$ is a projection band by \Cref{CharOfBands}.

The general case follows from the first part of the proof and that whenever $X$ is ``just" Hausdorff, then $C_b(X)$ and $C(\beta X/_\sim)$ are isometrically algebra and lattice isomorphic.
\end{proof}

The following corollary follows immediately from \Cref{bands vs projection bands} and \Cref{proj band in hausdorff}.

\begin{cor}\label{un vs isolated points}
Let $ X$ be a Hausdorff space and $x\in  X$ such that
$f_\alpha\goesun 0$ in $J_{\{x\}}(X)$ implies $f_\alpha\goesun 0$ in $C_b(X)$. Then $x$ is isolated in $ X.$
\end{cor}

We proceed  with an application of \Cref{closed ideals bands functionally hasudorff} and \Cref{bands vs projection bands}.

\begin{cor}
Let $X$ be a topological space with the property that for every maximal ideal $J$ in $C_b(X)$ $f_\alpha\goesun 0$ in $J$ implies $f_\alpha\goesun 0$ in $C_b(X)$.
\begin{enumerate}
\item [(a)] If $X$ is functionally Hausdorff, then $X$ is finite.
\item [(b)] If $X$ is Hausdorff, then $C_b(X)$ is finite-dimensional.
\end{enumerate}
\end{cor}

\begin{proof}
Since every point in $ X$ is isolated by \Cref{un vs isolated points}, $ X$ is discrete. Discrete compact spaces are obviously finite.
\end{proof}

We conclude this paper with $C_0(X)$-analogs of results of this section. Since the proofs are basically the same as in the $C_b(X)$-case, we omit them.

\begin{thm}
Let $X$ be a locally compact Hausdorff space.
\begin{enumerate}
\item [(a)] Let $J$ be a closed ideal in $C_0(X)$. If every un-null net $(f_\alpha)$ in $J$ is also un-null in $C_0(X)$, then $J$ is a projection band in $C_0(X)$.
\item [(b)] Pick $x\in X$. If $f_\alpha\goesun 0$ in $J_{\{x\}}^0(X)$ implies $f_\alpha\goesun 0$ in $C_0(X)$, then $x$ is isolated in $X.$
\item [(c)] If for every maximal ideal $J$ in $C_0(X)$ $f_\alpha\goesun 0$ in $J$ implies $f_\alpha\goesun 0$ in $C_0(X)$, then $X$ is discrete.
\end{enumerate}
\end{thm}

%
%

\end{document}